\documentclass[12pt]{amsart}
\usepackage[utf8]{inputenc}
\usepackage[english]{babel} 

\usepackage{amssymb}
\usepackage{amsthm}  
\usepackage{amsmath} 
\usepackage{enumerate} 
\usepackage{mathrsfs}
\usepackage{upgreek} 
\usepackage{fix-cm} 
\usepackage{adjustbox}
\usepackage{verbatim} 
\usepackage{stmaryrd} 
\usepackage{dirtytalk} 
\usepackage{tikz}

\usepackage{kantlipsum} 

\calclayout

\usepackage{tikz-cd}

\usepackage{aliascnt}

\usepackage{graphicx}

\makeatletter
\def\l@subsection{\@tocline{2}{0pt}{2.5pc}{5pc}{}}
\makeatother

\DeclareRobustCommand{\SkipTocEntry}[5]{}

\usepackage{xcolor} 
\usepackage[colorlinks=true, linkcolor=blue, urlcolor=red, citecolor=brown]{hyperref} 

\addto\extrasenglish{

  }

\let\oldfootnotemark\footnotemark
\let\oldfootnotetext\footnotetext
\let\oldfootnote\footnote
\renewcommand\footnote[1]{\addtocounter{footnote}{1}\hypertarget{fnbackref.\arabic{footnote}}{}\addtocounter{footnote}{-1}\oldfootnote{#1\fnbackref}}
\renewcommand\footnotemark{\addtocounter{footnote}{1}\hypertarget{fnbackref.\arabic{footnote}}{}\addtocounter{footnote}{-1}\oldfootnotemark}
\renewcommand\footnotetext[1]{\oldfootnotetext{#1\fnbackref}}
\newcommand{\fnbackref}{\hyperlink{fnbackref.\arabic{footnote}}{\footnotesize$\uparrow$}}

\theoremstyle{definition}
\newtheorem{dfn}{Definition}[subsection]

\theoremstyle{remark}

\newaliascnt{rmk}{dfn}
\newtheorem{rmk}[rmk]{Remark}
\aliascntresetthe{rmk}

\newaliascnt{ex}{dfn}
\newtheorem{ex}[ex]{Example}
\aliascntresetthe{ex}

\theoremstyle{plain}
\newaliascnt{thm}{dfn}
\newtheorem{thm}[thm]{Theorem}
\aliascntresetthe{thm}

\newaliascnt{prop}{dfn}
\newtheorem{prop}[prop]{Proposition}
\aliascntresetthe{prop}

\newaliascnt{prop2}{dfn}

\aliascntresetthe{prop2}

\newaliascnt{lem}{dfn}
\newtheorem{lem}[lem]{Lemma}
\aliascntresetthe{lem}

\newaliascnt{cor}{dfn}
\newtheorem{cor}[cor]{Corollary}
\aliascntresetthe{cor}

\newcommand{\mb}[1]{\mathbb{#1}}

\newcommand{\ca}[1]{\mathcal{#1}}

\newcommand{\mcr}[1]{\mathscr{#1}}

\makeatletter
\newcommand{\nodeequation}[1]{%
  \let\label\ltx@label
  \refstepcounter{equation}%
  #1
  \quad
  (\theequation)%
}
\makeatother

\begin{document}


\title{Entropy of the composition of two spherical twists}


\author{Federico Barbacovi}

\address{Department of Mathematics, University College London}
\email{federico.barbacovi.18@ucl.ac.uk}


\author{Jongmyeong Kim}
\address{Center for Geometry and Physics, Institute for Basic Science (IBS), Pohang 37673, Republic of Korea}
\email{myeong@ibs.re.kr}



\begin{abstract}
  Given a categorical dynamical system, {\it i.e.} a triangulated category together with an endofunctor, one can try to understand the complexity of the system by computing the entropy of the endofunctor.
  Computing the entropy of the composition of two endofunctors is hard, and in general the result doesn't have to be related to the entropy of the single pieces.

  In this paper we compute the entropy of the composition of two spherical twists around spherical objects, showing that it depends on the dimension of the graded vector space of morphisms between them.
  As a consequence of these computations we produce new counterexamples to Kikuta--Takahashi's conjecture.
  In particular, we describe the first counterexamples in odd dimension and examples for the $d$-Calabi--Yau Ginzburg dg algebra associated to the $A_2$ quiver.
\end{abstract}


\maketitle


\tableofcontents


\section{Introduction}



In \cite{DHKK-Dynamical-systems} the authors introduced the notion of a {\it categorical dynamical system}: a couple ($\mcr{T}$, $\Phi$) of a triangulated category together with an endofunctor $\Phi : \mcr{T} \rightarrow \mcr{T}$, and that of the {\it entropy} of an endofunctor: a function $h_t(\Phi): \mb{R} \rightarrow [-\infty, +\infty)$.

Since their introduction, these ideas have received a lot of attention and many people have made contributions to the subject.
However, computing explicit examples of entropies of endofunctors is a highly non-trivial task which has been accomplished only in a few cases, e.g. tensor product with lines bundles \cite{DHKK-Dynamical-systems}, spherical twists around spherical objects \cite{Ouchi-entropy-spherical-twist}, and $\mb{P}$-twist around $\mb{P}$-objects \cite{Fan-P-twists}.

Recently, in \cite{Kim-Entropy} the second author proved a theorem that relates the entropy of the twist around a spherical functor with that of (a shift of) the cotwist.
Such result potentially allows one to estimate the entropy of any autoequivalence as it is known that any autoequivalence is the spherical twist around a spherical functor, see \cite{Seg-autoeq-spherical-twist}.
Moreover, as a fixed autoequivalence can be realized as a spherical twist in many different ways, one can try to make the computations easier by choosing a good representation as a spherical twist.

In \cite{Barb-Spherical-twists} the first author described how to realize the composition of the twists around two spherical functors as a single twist, and therefore there is a natural candidate to which the above result can be applied in order to compute the entropy of the composition of two autoequivalences.

Even though these ideas seem to be profitable the general case is out of reach for the moment.
For this reason we concentrate on the case of the composition of two spherical twists around spherical objects, which already shows interesting features.

A detailed statement would require us to consider various different cases and it goes beyond the scope of this introduction.
Hence, we will content ourselves with an imprecise formulation.

\begin{thm}
	\label{thm:first-theorem-intro}
	If $E_1$ and $E_2$ are two $d$-spherical objects in a $k$-linear, proper, dg enhanced, triangulated category $\mcr{T}$ with a split-generator and a Serre functor, and $V := \textup{Hom}^{\bullet}_{\mcr{T}}(E_2, E_1)$ satisfies $\mathrm{Hom}_{D^b(k)}(V,V[d]) = 0$, then, we can explicitly compute or \say{precisely} bound $h_t(T_{E_2} \circ T_{E_1})$ when $\dim V = 0,1,2$.
\end{thm}

In contrast to the theory of entropy for dynamical systems from which it draws inspiration, the entropy of endofunctors naturally incorporates the dependence on a real variable $t \in \mb{R}$.
When evaluating the entropy at $0$, $h_0(\Phi)$, we speak of the {\it categorical entropy} of $\Phi$.

Even though we are not able to compute the entropy of $T_{E_2} \circ T_{E_1}$ for all values of $t \in \mb{R}$, we are able to compute its categorical entropy.
More precisely, we have

\begin{thm}
	With the same notation and assumptions as in \autoref{thm:first-theorem-intro}, we have
	\[
		h_0(T_{E_2} \circ T_{E_1}) = \left\{
		\begin{array}{ll}
			0 & \dim V = 0,1,2\\
			\displaystyle{\log \left(\frac{(\dim V)^2 - 2 + \sqrt{(\dim V)^4 - 4(\dim V)^2}}{2} \right)} > 0 & \dim V \geq 3
		\end{array}
		\right..
	\]
\end{thm}

The content of the above two theorems is summed up in \autoref{thm:entropy-composition-1} and \autoref{thm:entropy-composition-2}.

In \cite{Kikuta-Takahashi-On-cat-entropy} the authors proposed a conjecture that relates the categorical entropy of an autoequivalence with the spectral radius of the induced linear isomorphism on $K_{\mathrm{num}}(\mcr{T})$.
More precisely, if $\Phi : \mcr{T} \rightarrow \mcr{T}$ is an autoequivalence and $K_{\mathrm{num}}(\mcr{T})$ is the numerical Grothendieck group of $\mcr{T}$, then the conjecture says
\[
	h_0(\Phi) = \log \rho([\Phi]),
\]
where $[\Phi]: K_{\mathrm{num}}(\mcr{T}) \rightarrow K_{\mathrm{num}}(\mcr{T})$ is the induced map and $\rho([\Phi]) = \max \{ |\lambda| : \lambda$ eigenvalue of $[\Phi] \}$.
In \cite{KST-lower-bound} the authors proved the lower bound $\geq$, but since then counterexamples have been found, \cite{Fan-entropy-CY}, \cite{Ouchi-entropy-spherical-twist}, \cite{Mattei-categorical-entropy-surfaces}.

Using the above theorems we are able to give a numerical condition that ensures when Kikuta--Takahashi's conjecture holds for the composition of two spherical twists, see \autoref{cor:examples}.
In particular, we are able to produce the first counterexamples to Kikuta--Takahashi in odd dimension (as hypersurfaces in $\mb{P}^n \times \mb{P}^m$), see \autoref{ex:KT-PxP}, and examples in the subgroup $\langle T_{S_1}, T_{S_2} \rangle$ of $D(\Gamma_2^d)$, where $\Gamma_2^d$ is the d-Calabi-Yau Ginzburg dg algebra and $S_i$'s are the two spherical objects supported at the vertices, see \autoref{cor:ginzburg}.

The motivation of Fan's first counterexample to the Kikuta--Takahashi's conjecture was to find a mirror counterpart of Thurston's construction of a map on a surface with positive topological entropy acting trivially on homology, \cite{Fan-entropy-CY}.
The existence of a 4-dimensional example of such a map was shown recently in \cite{Kikuta-Ouchi}.
In \autoref{cor:milnor}, we give an interpretation of the $A_2$ Ginzburg dg algebra example in terms of symplectic geometry and see that certain compositions of Dehn twists give examples of such a map in even dimensions, see \autoref{rmk:torelli}.

\addtocontents{toc}{\SkipTocEntry}
\subsection*{Acknowledgments}

The authors would like to thank Kohei Kikuta and Ed Segal for reading a draft of this preprint and providing many helpful suggestions.
F.B. would like to thank his advisor Ed Segal for many helpful conversations.
F.B. was supported by the European Research Council (ERC) under the European Union Horizon 2020 research and innovation programme (grant agreement No.725010).
J.K. was supported by the Institute for Basic Science (IBS-R003-D1).

\section{Entropy of the spherical twist around a spherical functor}



Let $\mathscr{T}$ be a $k$-linear triangulated category.
In this paper, we study a {\em categorical dynamical system}, {\it i.e.} a triangulated category together with an endofunctor.
To study the complexity of a categorical dynamical system, \cite{DHKK-Dynamical-systems} introduced the notion of {\em categorical entropy}.

\begin{dfn}
	For $E,F \in \mathscr{T}$, the {\em categorical complexity} of $F$ with respect to $E$ is the function $\delta_t(E,F) : \mathbb{R} \to [0,\infty]$ given by
	\[
		\delta_t(E,F) = \inf
		\left\{\sum_{i=1}^k e^{n_i t} \,\left|\,
		\begin{tikzcd}[column sep=tiny]
			0 \ar[rr] & & * \ar[dl] \ar[rr] & & * \ar[dl] & \cdots & * \ar[rr] & & F \oplus F' \ar[dl]\\
			& E[n_1] \ar[ul,"+1"] & & E[n_2] \ar[ul,"+1"] & & & & E[n_k] \ar[ul,"+1"] &
		\end{tikzcd}
		\right.\right\}
	\]
	if $F \not\cong 0$, and $\delta_t(E,F) = 0$ if $F \cong 0$.
	Here the infimum is taken over all possible cone decompositions of objects of the form $F \oplus F'$ into $E[n_i]$'s.
\end{dfn}

An object $G \in \mathscr{T}$ is called a {\em split-generator} if the smallest full triangulated subcategory containing $G$ and closed under taking direct summands coincides with $\mathscr{T}$ itself.

\begin{dfn}
	Let $G$ be a split-generator of $\mathscr{T}$.
	The {\em categorical entropy} of an exact endofunctor $\Phi : \mathscr{T} \to \mathscr{T}$ is the function $h_t(\Phi) : \mathbb{R} \to [-\infty,\infty)$ given by
	\[
		h_t(\Phi) = \lim_{n \to \infty} \frac{1}{n} \log \delta_t(G,\Phi^n(G)).
	\]
\end{dfn}

\begin{rmk}
	The categorical entropy is well-defined, {\it i.e.} the limit exists in $[-\infty,\infty)$ and does not depend on the choice of a split-generator \cite{DHKK-Dynamical-systems}.
	Moreover it can be also written as
	\[
		h_t(\Phi) = \lim_{n \to \infty} \frac{1}{n} \log \delta_t(G,\Phi^n(G'))
	\]
	for any choice of split-generators $G,G'$ of $\mathscr{T}$, see \cite{Kikuta-Curve}.
\end{rmk}

Let $\mathscr{D},\mathscr{T}$ be $k$-linear triangulated categories with dg enhancements.

\begin{dfn}
	An exact functor $f : \mathscr{D} \to \mathscr{T}$ with right and left adjoint functors $f^R,f^L$ is called a {\em spherical functor} if it satisfies the following conditions:
	\begin{itemize}
		\item[(1)] The {\em twist functor} $T_f = \mathrm{cone}(ff^R \overset{\varepsilon}{\to} \mathrm{Id}_\mathscr{T})$ is an exact autoequivalence of $\mathscr{T}$, where $\varepsilon : ff^R \to \mathrm{Id}_\mathscr{T}$ is the counit of the adjoint pair $f \dashv f^R$.
		\item[(2)] The {\em cotwist functor} $C_f = \mathrm{cone}(\mathrm{Id}_\mathscr{D} \overset{\eta}{\to} f^Rf)[-1]$ is an exact autoequivalence of $\mathscr{D}$, where $\eta : \mathrm{Id}_\mathscr{D} \to f^Rf$ is the unit of the adjoint pair $f \dashv f^R$.
		\item[(3)] $f^R \cong f^LT_f[-1]$.
		\item[(4)] $f^R \cong C_ff^L[1]$.
	\end{itemize}
\end{dfn}

In {\cite[Theorem 1.6, 1.7]{Kim-Entropy}}, the second author proved the following theorem which relates the entropy of the twist with that of the cotwist.

\begin{thm}
	\label{thm:entropy-twist-spherical-functor}
	Let $f : \mathscr{D} \to \mathscr{T}$ be a spherical functor with right adjoint functor $f^R$.
	\begin{itemize}
		\item[(1)] Assume that the essential image of $f^R$ contains a split-generator of $\mathscr{D}$.
		Then
			\[
			h_t(C_f[2])	\leq h_t(T_f) \leq
			\begin{cases}
			0 & \text{for every } t \text{ such that } h_t(C_f[2]) \leq 0,\\
			h_t(C_f[2]) & \text{for every } t \text{ such that } h_t(C_f[2]) \geq 0.
			\end{cases}
			\]
		\item[(2)] Assume that $\mathrm{Ker}\, ff^R \neq 0$.
		Then
			\[
			h_t(T_f) \geq 0.
			\]
	\end{itemize}
\end{thm}

\begin{ex}
	This theorem can be considered as a generalization of the computations of the entropy of the spherical twist \cite{Ouchi-entropy-spherical-twist} and the $\mathbb{P}$-twist \cite{Fan-P-twists}.
	
	In fact, if $E$ is a $d$-spherical object ($d \geq 1$) in $\mathscr{T}$, the functor $f = - \otimes_k E : D^b(k) \to \mathscr{T}$ is spherical and
	\[
	T_f \cong T_E, \quad \quad C_f \cong [-1-d]
	\]
	where $T_E$ denotes the {\em spherical twist} around $E$.
	Since in $D^b(k)$, by \cite[Theorem 2.6]{DHKK-Dynamical-systems}, we have $h_t([m]) = mt$ for any $m \in \mathbb{Z}$ , \autoref{thm:entropy-twist-spherical-functor} implies that
		\[
		(1-d)t	\leq h_t(T_E) \leq
		\begin{cases}
		0 & t \geq 0,\\
		(1-d)t & t \leq 0.
		\end{cases}
		\]
	Moreover, it also implies that if $E^\perp := \{ F \in \mathscr{T} \,|\, \mathrm{Hom}_\mathscr{T}^\bullet(E,F) = 0 \} \neq 0$, then $h_t(T_E) = 0$ for all $t \geq 0$.
	This is exactly the main result of \cite{Ouchi-entropy-spherical-twist}.
	
	The main result of \cite{Fan-P-twists} can be obtained similarly using a presentation of a $\mathbb{P}$-object as a spherical functor \cite{Seg-autoeq-spherical-twist}.
\end{ex}

In general, it is not easy to verify the technical conditions of the above theorem.
However, the following lemma from \cite{Kim-Entropy} provides a useful sufficient condition for the condition of part (1) of \autoref{thm:entropy-twist-spherical-functor}.

\begin{lem}
	\label{lem:tech-condition-dt-triangle}
	Let $f : \mathscr{D} \to \mathscr{T}$ be a spherical functor with right adjoint functor $f^R$ and $G$ be a split-generator of $\mathscr{D}$.
	Assume that there is an integer $n > 0$ such that $\mathrm{Hom}_\mathscr{D}(C_f^n(G),G) = 0$.
	Then $f^Rf(G \oplus C_f(G) \oplus \cdots \oplus C_f^{n-1}(G))$ is a split-generator of $\mathscr{D}$.
\end{lem}

\begin{proof}
	Set $X_1 = f^Rf(G)$.
	It sits in the exact triangle
	\begin{equation}
		\label{dt-cotwist}
		\begin{tikzcd}
		G \ar[r,"\eta_G"] & f^Rf(G) \ar[r,"\phi_1"] & C_f(G)[1] \ar[r] & G[1]
		\end{tikzcd}
	\end{equation}
	defining the cotwist functor $C_f$.
	Then, we inductively define a sequence $\{X_n\}_{n=1}^\infty$ of objects of $\mathscr{D}$ by the commutative diagram
	\[
		\begin{tikzcd}
			X_{n-1}[-1] \ar[r,"\phi_{n-1}{[-1]}"] \ar[d,equal] & C_f^{n-1}(G) \ar[d,"\eta_{C_f^{n-1}(G)}"] \ar[r] & G \ar[r] \ar[d] & X_{n-1} \ar[d,equal]\\
			X_{n-1}[-1] \ar[r] & f^RfC_f^{n-1}(G) \ar[r] \ar[d] & X_n \ar[r] \ar[d,"\phi_n"] & X_{n-1}\\
			& C_f^n(G)[1] \ar[r,equal] \ar[d] & C_f^n(G)[1] \ar[d] &\\
			& C_f^{n-1}(G)[1] \ar[r] & G[1]. &
		\end{tikzcd}
	\]
	obtained by applying the octahedral axiom.

	The assumption implies that the exact triangle
	\[
		\begin{tikzcd}
		G \ar[r] & X_n \ar[r] & C_f^n(G)[1] \ar[r] & G[1]
		\end{tikzcd}
	\]
	splits, and therefore $X_n \cong G \oplus C_f^n(G)[1]$.
	The lemma follows since $X_n$ is split-generated by $f^Rf(G \oplus C_f(G) \oplus \cdots \oplus C_f^{n-1}(G))$ by construction.
\end{proof}

\begin{cor}
	Let $f : \mathscr{D} \to \mathscr{T}$ be a spherical functor with right adjoint functor $f^R$.
	If there exists a split-generator $G$ of $\mathscr{D}$ and an integer $n > 0$ such that $\mathrm{Hom}_\mathscr{D}(C_f^n(G),G) = 0$, then the essential image of $f^Rf$ contains a split-generator of $\mathscr{D}$.
\end{cor}

In general, computing the entropy of an endofunctor is a very hard task, and we will try to tackle this question using \autoref{thm:entropy-twist-spherical-functor}.
However, there is a case in which we can bound the entropy using its value at zero and some asymptotic behaviour.

\begin{prop}[{\cite[Theorem 2.1.7]{Shifting-numbers}, \cite[Proposition 6.13, 6.14]{Elagin-Lunts-dimension}}]
	\label{prop:bound-entropies}
	For any non-nilpotent endofunctor $F$ of $\mcr{T} = D(T)^c$, $T$ a smooth and compact dg algebra, the limits
	\[
		\lim_{t \to \pm \infty} \frac{h_t(F)}{t} = \tau^{\pm}(F)
	\]
	are finite and we have the inequalities
	\[
		\renewcommand{\arraystretch}{1.5}
		\begin{array}{lcr}
			\tau^{+}(F) t \leq h_t(F) \leq h_0(F) + \tau^{+}(F)t & & t \geq 0,\\
			\tau^{-}(F) t \leq h_t(F) \leq h_0(F) + \tau^{-}(F)t & & t \leq 0.
		\end{array}
	\]
\end{prop}

\section{Upper triangular dg algebras and gluing}



Let us consider two dg algebras $A$, $B$ and an $A-B$ bimodule $V$.
From this data we can construct a new dg algebra $R := B \oplus A \oplus V$, where the grading and the differential are defined componentwise, and the multiplication is $(b,a,v) \cdot (b', a', v') = (bb', aa', vb' + av')$.

This new dg algebra is sometimes denoted
\[
	R = \left( \begin{array}{cc} A & V\\ 0 & B\end{array}\right),
\]
and is called an upper triangular dg algebra.

In \cite{Barb-Spherical-twists} the first author used such a dg algebra to represent the composition of the spherical twists around two spherical objects as the spherical twist around a single spherical functor.

Let us briefly recall this construction.
Consider $\mcr{T}$ a $k$-linear, proper, dg enhanced triangulated category with a split-generator and a Serre functor $\mb{S}_{\mcr{T}}$.
Take $E_1$ and $E_2$ two $d$-spherical objects in $\mcr{T}$, {\it i.e.} they satisfy
\[
	\mb{S}_{\mcr{T}} E_i \simeq E_i[d], \quad \quad \text{Hom}^{\bullet}_{\mcr{T}}(E_i, E_i) := \bigoplus_{n \in \mb{Z}} \text{Hom}_{\mcr{T}}(E_i, E_i[n]) \simeq k[t] \left/ t^2 \right., \; \deg{t} = d,
\]
where the second isomorphism is of graded algebras.

Then, we can construct the autoequivalence
\[
	T_i(F) := \text{cone} \left( \text{Hom}_{\mcr{T}}^{\bullet}(E_i, F) \otimes E_i \rightarrow F \right)
\]
called the {\it spherical twist} around $E_i$, see \cite{Seidel-Thomas01}.

By using a dg enhancement of $\mcr{T}$ we can fix a dg $k$-module $W$ whose associated graded module $H^{\bullet}(W)$ is isomorphic to $V := \text{Hom}^{\bullet}_{\mcr{T}}(E_2,E_1)$.
Then, we can define the upper triangular dg algebra $R = k \oplus k \oplus W$ and consider $E_2 \oplus E_1$ as a left dg module over $R$.
Here the two copies of $k$ act on the left and on the right on $W$ via the identity of (the dg lift of) $E_1$ and $E_2$, respectively.
Notice however that such an upper triangular dg algebra is formal because we can write an explicit quasi isomorphism $H^{\bullet}(R) \rightarrow R$ by choosing representatives of the cohomology classes of $W$.
In particular, the dg enhancement of $\mcr{T}$ doesn't matter in this particular construction, and we directly consider the graded algebra $R := k \oplus k \oplus V$.

With these remarks in mind, \cite[Theorem 3.2.1]{Barb-Spherical-twists} can be stated as follows

\begin{thm}
	\label{thm:composition-twists}
	The left $R$ module $E_2 \oplus E_1$ defines a spherical functor
	\[
		\begin{tikzcd}[column sep = 7em]
			D(R)^c \ar[r, "f := - \stackrel{L}{\otimes}_R (E_2 \oplus E_1)"] & \mcr{T}
		\end{tikzcd}
	\]
	whose twist is given by $T_f \simeq T_{2} \circ T_1$ and whose cotwist is given by $C_f \simeq - \stackrel{L}{\otimes}_{R} R^{\ast}[-1-d]$.
\end{thm}

\begin{rmk}
	In \cite{Barb-Spherical-twists} the cotwist was described for the dg algebra and not for its associated graded algebra.
	The description of the cotwist in the above formulation follows from the fact that if $A \rightarrow B$ is a quasi isomorphism of dg algebras, then the dual map $B^{\ast} \rightarrow A^{\ast}$ is quasi isomorphism of $A-A$ dg bimodules.
\end{rmk}

In particular, as $R$ is smooth and proper, we see that the cotwist gives Serre duality on $D(R)^c$ up to a shift, see \cite{shklyarov2007serre}.

\subsection{A distinguished triangle}

Our aim is now to give sufficient conditions under which the technical condition of part (1) of \autoref{thm:entropy-twist-spherical-functor} is verified for the case of the composition of two spherical twists around spherical objects.

As a consequence of \autoref{thm:composition-twists}, we get the distinguished triangle of right $R$ dg modules
\[
	\begin{tikzcd}
		R \ar[r] & \text{RHom}_{\mcr{T}}(E_2 \oplus E_1, E_2 \oplus E_1) \ar[r] & R^{\ast}[-d] \ar[r] & R[1],
	\end{tikzcd}
\]
where $\text{RHom}_{\mcr{T}}(E_2 \oplus E_1, E_2 \oplus E_1)$ denotes the dg endomorphism algebra of (a dg lift of) $E_2 \oplus E_1$ in some dg enhancement of $\mcr{T}$.
This is triangle \eqref{dt-cotwist} for the spherical functor of \autoref{thm:composition-twists}.

By \autoref{lem:tech-condition-dt-triangle} we know that to satisfy the technical condition of \autoref{thm:entropy-twist-spherical-functor} is enough to prove 
\[
	0 = \text{Hom}_{D(R)}(R^{\ast}[-d], R[1]) \simeq H^{1+d}(R^{!}), \quad R^{!} = \text{RHom}_{D(R-R)}(R, R \otimes_k R).
\]

Recall that $V = \text{Hom}_{\mcr{T}}^{\bullet}(E_2, E_1)$.
We have

\begin{lem}
	\label{lem:vanishing-R-!}
	If
	\[
		V^{1+d} = (V^{\ast})^d = (V \otimes_k V^{\ast})^d = ( V^{\ast} \otimes_k V \otimes_k V)^{d} = 0,
	\]
	then $H^{1+d}(R^{!}) = 0$.
\end{lem}

\begin{proof}
	For clarity let us denote $k_1$ the copy of $k$ acting on $R$ via $\text{id}_{E_1}$ and $k_2$ the one acting via $\text{id}_{E_2}$.
	Then, by the definition of $R$ we have the distinguished triangle of $R-R$ bimodules
	\[
	\begin{tikzcd}
		R \otimes_{k_1} V \otimes_{k_2} R \ar[r] & R \otimes_{k_1} R \oplus R \otimes_{k_2} R \ar[r] & R \ar[r] & R \otimes_{k_1} V \otimes_{k_2} R [1].
	\end{tikzcd}
	\]
	Using this distinguished triangle, we see that we have the distinguished triangle
	\[
		\begin{tikzcd}
			R^{!} \ar[r] & \begin{array}{c}\text{RHom}_{R-R}(R \otimes_{k_1} R, R \otimes_k R)\\ \oplus\\ \text{RHom}_{R-R}(R \otimes_{k_2} R, R \otimes_k R)\end{array} \ar[r] & \text{RHom}_{R-R}(R \otimes_{k_1} V \otimes_{k_2} R, R \otimes_k R) \ar[r] & R^![1].
		\end{tikzcd}
	\]

	Now notice that\footnote{Here the subscript means that we are restring the action via the inclusion $k_i \hookrightarrow R$.}
	\[
		{}_{k_1} R \simeq k_1 \oplus V \quad {}_{k_2} R \simeq k_2 \quad R_{k_1} \simeq k_1 \quad R_{k_2} \simeq V \oplus k_2.
	\]
	Using these isomorphisms of bimodules we can simplify the above distinguished triangle and get 
	\[
		\begin{tikzcd}
			R^{!} \ar[r] & k \oplus k \oplus V \oplus V \ar[r] & V^{\ast} \oplus V^{\ast} \otimes V \oplus V^{\ast} \otimes V \oplus V^{\ast} \otimes V \otimes V \ar[r] & R^![1].
		\end{tikzcd}
	\]

	Then, the statement follows from taking the long exact sequence induced by the above distinguished triangle.
\end{proof}

We now wish to show that all of the conditions of \autoref{lem:vanishing-R-!} can be achieved if $V$ satisfies\footnote{Notice that $V$ is bounded by construction as $\mcr{T}$ is proper.} $\mathrm{Hom}_{D^b(k)}(V, V[d]) = 0$.

Indeed, if we call $\max V$ and $\min V$ the maximum and the minimum degree respectively of a non zero element of $V$, then $\max V \leq d$ implies $V^{1+d} = 0$, $\max V < -d$ implies $(V^{\ast})^d = 0$, and $2 \max V - \min V < d$ implies $(V^{\ast} \otimes_k V \otimes_k V)^d = 0$.
Now notice that if we exchange $E_1$ with $E_1[n]$ the spherical twist doesn't change: $T_{E_1[n]} \simeq T_1$, but the degrees in which $V_n = \text{Hom}^{\bullet}_{\mcr{T}}(E_2, E_1[n])$ lives do.
More precisely, we have
\[
	\max V_n = \max V - n \quad \quad 2 \max V_n - \min V_n = 2 \max V - \min V - n.
\]
In particular, if we take $n >> 0$ the three inequalities above can always be achieved, and the only remaining vanishing required by \autoref{lem:vanishing-R-!} is $(V^{\ast} \otimes_k V)^d = \mathrm{Hom}_{D^b(k)}(V,V[d]) = 0$.

Hence we get

\begin{lem}
	\label{lem:inequalities}
	Let $E_1$ and $E_2$ be two $d$-spherical objects in $\mcr{T}$ a $k$-linear, proper, dg enhanced triangulated category with a split-generator and a Serre functor.
	Set $V := \textup{Hom}^{\bullet}_{\mcr{T}}(E_2, E_1)$.

	If $\mathrm{Hom}_{D^b(k)}(V,V[d]) = 0$, then, up to replacing $E_1$ with $E_1[n]$ and $V$ with $V[n]$ for $n>>0$, the assumption of part (1) of \autoref{thm:entropy-twist-spherical-functor} are satisfied for the spherical functor
	\[
		\begin{tikzcd}[column sep = 7em]
			D(R)^c \ar[r, "f := - \stackrel{L}{\otimes}_R (E_2 \oplus E_1)"] & \mcr{T}.
		\end{tikzcd}
	\]
\end{lem}

Let us remark that we don't know whether the condition $\mathrm{Hom}_{D^b(k)}(V,V[d]) = 0$ is really needed or whether it can be removed by a more thorough study of the map $R^{\ast}[-d] \rightarrow R[1]$.

\begin{rmk}
	In principle what we did in this section can be done for any upper triangular dg algebra, and hence one could try to find sufficient conditions under which the hypothesis of part (1) of \autoref{thm:entropy-twist-spherical-functor} is satisfied for any couple of spherical functors.

	Unfortunately, the problem is that the terms involved are now $\text{RHom}$'s between dg bimodules over the dg algebras $A$ and $B$ from which the upper triangular dg algebra $R$ is constructed.
	Hence, homs can go in any direction regardless of the cohomological bounds we impose.

	However, it is worthy to point out that in the case of the dg algebra arising from \cite[Theorem 4.1.2]{Barb-Spherical-twists} for the composition of many spherical twists around spherical objects, it is still possible to give sufficient conditions based on cohomological bounds (because we can bring all the $\text{RHom}$ back at the vertices of the dg algebra).
\end{rmk}

\subsection{Categorical entropy of the Serre functor}

\autoref{thm:composition-twists} and \autoref{lem:vanishing-R-!} tell us that if $\mathrm{Hom}_{D^b(k)}(V,V[d]) = 0$, $V = \text{Hom}_{\mcr{T}}^{\bullet}(E_2, E_1)$, then the entropy of $T_2 \circ T_1$ can be computed using the entropy of the Serre functor for $D(R)^c$, $R = k \oplus k \oplus V$ (up to shift $V$, but we won't care about this because shifting $V$ won't affect the final result, as it ought to be).

Even though our motivation for computing the entropy of the Serre functor of $D(R)^c$ is computing the entropy of $T_2 \circ T_1$, the result of this section apply for any upper triangular dg algebra of the form $k \oplus k \oplus W$ where $W$ is a graded vector space.
Hence, in the following $A$ will denote any upper triangular dg algebra of the form $A = k \oplus k \oplus W$.

We know by \cite{shklyarov2007serre} that the Serre functor for $D(A)^c$ is given by $\mb{S}_A := - \stackrel{L}{\otimes}_{A} A^{\ast}$, so the only thing we have to do is to compute $h_t(\mb{S}_A)$.

Unfortunately, this is not an easy task for a general $t \in \mb{R}$.
However, using the results of \cite{Elagin-examples-dimensions}, we will be able to compute the categorical entropy of $\mb{S}_A$, {\it i.e.} $h_0(\mb{S}_A)$.

\begin{rmk}
	In \cite[pag. 32]{DHKK-Dynamical-systems} the authors state the value of the entropy of the Serre functor for the derived category of the Kronecker quiver with $m \geq 3$ arrows. Our computations will recover that value when $W$ lives only in degree $0$, and they will show that the grading on $W$ doesn't affect $h_0(\mb{S}_A)$.
\end{rmk}

Assume from now on $\dim W \geq 2$.
As our category is of the form $D(A)^c$, by \cite[Theorem 2.6]{DHKK-Dynamical-systems} we know that computing $h_0(\mb{S}_A)$ amounts to computing
\[
	\lim_{m \to +\infty} \frac{1}{m} \log \left( \sum_{n \in \mathbb{Z}} \dim H^n \left( (A^{\ast})^{\otimes_A m} \right)\right).
\]
Thanks to\footnote{What we denote $d_m$ is $\dim \psi_{m}(W)$ in {\it ibidem}.} \cite[Lemma 8.2]{Elagin-examples-dimensions}, we know that
\[
	\sum_{n \in \mathbb{Z}} \dim H^n \left( (A^{\ast})^{\otimes_A m} \right) = d_{2m-2} + d_{2m-3} + d_{2m-1} + d_{2m-2},
\]
where $d_m$ satisfies the relations
\begin{equation}
	\label{recurrence-eq}
	\renewcommand{\arraystretch}{1.2}
	\begin{array}{l}
		d_{m+2} + d_{m} = d_{m+1} \cdot \dim W \quad \forall \; m \geq -1\\
		d_{1} = \dim W\\
		d_{0} = 1\\
		d_{-1} = 0.
	\end{array}
\end{equation}

Set $N = \dim W$.
To solve this recurrence relation we use the characteristic equation
\[
	N \sigma^{-1} - \sigma^{-2} = 1 \iff \sigma_{\pm} = \frac{N \pm \sqrt{N^2-4}}{2}.
\]
We see that we have to distinguish between two cases.

If $N=2$ the the solution to the recurrence equation is given by
\[
	d_m = m +1.
\]

If $N \geq 3$ then the solution is given by
\[
	d_m = \alpha \sigma_{-}^{m} + \beta \sigma_{+}^{m}, \quad \alpha = \frac{1}{2} - \frac{N}{2\sqrt{N^2 - 4}}, \quad \beta = \frac{1}{2} + \frac{N}{2\sqrt{N^2 - 4}}.
\]

Hence we get

\begin{lem}
	\label{lem:entropy-at-0-Serre}
	We have
	\[
		h_0( \mb{S}_A) = \left\{
		\begin{array}{lc}
			0 & \dim W = 2\\
			\displaystyle{\log \left(\frac{(\dim W)^2 - 2 + \sqrt{(\dim W)^4-4(\dim W)^2}}{2}\right)} > 0 & \dim W \geq 3
		\end{array}
		\right..
	\]
\end{lem}

\begin{proof}
	Notice that by the recurrence relations \eqref{recurrence-eq} we have
	\[
		\sum_{n \in \mathbb{Z}} \dim H^n \left( (A^{\ast})^{\otimes_A m} \right) = (2 + \dim W) d_{2m-2}.
	\]
	Hence we have
	\[
		h_0( \mb{S}_A) = \lim_{m \to +\infty} \frac{1}{m} \log (d_{2m-2}).
	\]

	If $N=2$ we have
	\[
		h_0( \mb{S}_A) = \lim_{m \to +\infty} \frac{1}{m} \log (2m-1) = 0.
	\]

	If $N \geq 3$ have
	\[
		\begin{aligned}
			h_0( \mb{S}_A) & \, = \lim_{m \to +\infty} \frac{1}{m} \log (\alpha \sigma_{-}^{2m-2} + \beta \sigma_{+}^{2m-2}) \\
			& \, =  \lim_{m \to +\infty} \frac{1}{m} (2m-2)\log ( \sigma_{+} ) = \log(\sigma_{+}^2),
		\end{aligned}
	\]
	where in the second line we used that $\sigma_{+} > \sigma_{-}$ and $\beta \neq 0$.
\end{proof}

\section{Composition of two spherical twists around spherical objects}



Now that we have introduced all the pieces that we need, we can move on to compute the categorical entropy of the composition of two spherical twists around spherical objects, and in some cases the entropy itself.

Let us recall the setting.
We have $\mcr{T} $ a $k$-linear, proper, dg enhanced triangulated category with a split-generator and a Serre functor.
Moreover, we have two $d$-spherical objects $E_1$, $E_2 \in \mcr{T}$, and we want to compute the entropy of $T_2 \circ T_1$, where $T_i = T_{E_i}$.

By \autoref{thm:composition-twists} we know that for $f = - \stackrel{L}{\otimes}_R (E_2 \oplus E_1): D(R)^c \rightarrow \mcr{T}$, where $R = k \, \text{id}_{E_2} \oplus k \, \text{id}_{E_1} \oplus \text{Hom}_{\mcr{T}}^{\bullet}(E_2, E_1)$, we have $T_2 \circ T_1 \simeq T_f$, $C_f \simeq \mb{S}_R[-1-d]$.

Moreover, by \autoref{lem:inequalities} we know that if $\mathrm{Hom}_{D^b(k)}(V,V[d]) = 0$, $V = \text{Hom}_{\mcr{T}}^{\bullet}(E_2, E_1)$, then we can compute the entropy of $T_2 \circ T_1$ using \autoref{thm:entropy-twist-spherical-functor}.

Finally, by \autoref{lem:entropy-at-0-Serre} we know the exact value $h_{0}(\mb{S}_R)$ when $\dim V \geq 2$.

Let us put together all these pieces to get the following results.

\begin{thm}
	\label{thm:entropy-composition-1}
	Assume $\mathrm{Hom}_{D^b(k)}(V,V[d]) = 0$, then the categorical entropy of $T_2 \circ T_1$ is given by
		\[
			\renewcommand{\arraystretch}{1.5}
			h_0(T_2 \circ T_1) = 
			\left\{
				\begin{array}{cl}
				0 & \dim V = 0,1,2\\
				\displaystyle{\log \left( \frac{(\dim V)^2 - 2 + \sqrt{(\dim V)^4 -4(\dim V)^2}}{2}\right)} > 0 & \dim V \geq 3
				\end{array}
			\right..
		\]
	Moreover, if $\dim V = 0$ we have
	\[
		h_t(T_{2} \circ T_{1}) = \left\{
			\begin{array}{lr}
				(1-d)t & t \leq 0\\
				\leq 0 & \mathrm{otherwise}
			\end{array}
		\right.,
	\]
	while if $\dim V = 1$ we have
	\[
		\renewcommand{\arraystretch}{1.5}
		h_t(T_{2} \circ T_{1}) = \left\{
			\begin{array}{lr}
				\displaystyle{\left( \frac{4}{3}-d \right)t} & \displaystyle{\forall t :\left( \frac{4}{3}-d \right)t \geq 0}\\
				\leq 0 & \mathrm{otherwise}
			\end{array}
		\right..
	\]
	In all cases, if $E_1^{\perp} \cap E_2^{\perp} \neq 0$, then $h_t(T_2 \circ T_1) \geq 0$, and in particular it is identically zero as soon as it is non-positive.
\end{thm}

\begin{proof}
	The first statement is a rephrasing of \autoref{lem:entropy-at-0-Serre} taking into account $C_f = \mb{S}_R[1-d]$.

	If $\dim V = 0$ we have $R = k \oplus k$ and $\mb{S}_R = \mathrm{id} \oplus \mathrm{id}$ on $D(R)^c \simeq D(k)^c \oplus D(k)^c$.
	Hence, using \cite[Theorem 2.6]{DHKK-Dynamical-systems} to compute $h_t(\mb{S}_R)$, we have
	\[
		h_t(T_{2} \circ T_{1}) = \underbrace{h_t(\mb{S}_R)}_{= 0} + (1-d)t = (1-d)t \quad \forall t : (1-d)t \geq 0 
	\]
	and $h_t(T_{2} \circ T_{1}) \leq 0$ otherwise.

	If $\dim V = 1$ we can always shift $V$ so that $R$ is the path algebra of the Dynkin quiver $A_2$.
	Hence, $D(R)^c$ is fractional Calabi--Yau of dimension $1/3$, see \cite{CY-categories-Keller}, \cite{CDIM-CY-algebras}.
	For any $d \geq 1$ \autoref{lem:inequalities} applies (without shifting $V$), and therefore, using once again \cite[Theorem 2.6]{DHKK-Dynamical-systems}, we get
	\[
		h_t(T_{2} \circ T_{1}) = h_t(\mb{S}_R) + (1-d)t = \left( \frac{4}{3}-d \right)t \quad \forall t :\left( \frac{4}{3}-d \right)t \geq 0
	\]
	and $h_t(T_{2} \circ T_{1}) \leq 0$ otherwise.
	The statement about the case in which the common orthogonal is not zero follows from \autoref{thm:entropy-twist-spherical-functor}.
\end{proof}

\begin{rmk}
	When $\dim V = 0$ the twists $T_2$ and $T_1$ commute with each other.
	In this case the result we obtained can also be proved using the same strategy used in \cite[Theorem 3.1]{Ouchi-entropy-spherical-twist}.
\end{rmk}

\begin{rmk}
	It was noticed in \cite[Theorem 3.1]{Ouchi-Hyperkahler} and \cite[Remark 3.5]{Mattei-categorical-entropy-surfaces} that the composition of many spherical twists can have positive categorical entropy, but the value of the entropy was not computed.
	The above theorem gives the precise value of the entropy of the composition of two spherical twists and tells us when it is positive.
\end{rmk}

\begin{thm}
	\label{thm:entropy-composition-2}
	Assume $\mathrm{Hom}_{D^b(k)}(V,V[d]) = 0$ and set $w = \max V - \min V$.
	If $w = 0$, then
		\[
			h_t(T_{2} \circ T_{1}) =
			\left\{
				\begin{array}{lc}
					(2-d)t & \mathrm{if} \; \, (2-d)t\geq 0\\
					\leq 0 & \mathrm{otherwise}
				\end{array}
			\right..
		\]

	Futhermore, if $\mcr{T} = D(T)^c$ for a smooth, compact dg algebra $T$ and $\dim V = 2$, then we have the following:
	\begin{enumerate}
		\item if $d+w \geq 2$ and $d-w > 2$, then
		\[
			h_t(T_{2} \circ T_{1}) =
			\left\{
				\begin{array}{cc}
					(2-(d+w))t & t \leq 0\\
					\leq 0 & t \geq 0
				\end{array}
			\right.;
		\]
		\item if $d+w \geq 2$ and $d-w \leq 2$, then
		\[
			h_t(T_{2} \circ T_{1}) =
			\left\{
				\begin{array}{cc}
					(2-(d+w))t & t \leq 0\\
					(2-(d-w))t & t \geq 0
				\end{array}
			\right.;
		\]
		\item if $d+w < 2$ and $d-w > 2$, then
		\[
			h_t(T_{2} \circ T_{1}) \leq 0 \quad \forall t \in \mb{R};
		\]
		\item if $d+w < 2$ and $d-w \leq 2$, then
		\[
			h_t(T_{2} \circ T_{1}) =
			\left\{
				\begin{array}{cc}
					\leq 0 & t \leq 0\\
					(2-(d-w))t & t \geq 0
				\end{array}
			\right.;
		\]
	\end{enumerate}
	In all cases, if $E_1^{\perp} \cap E_2^{\perp} \neq 0$, then $h_t(T_2 \circ T_1) \geq 0$, and in particular it is identically zero as soon as it is non-positive.
\end{thm}

\begin{proof}
	For the statement about $w = 0$ notice that in such case we can always shift $V$ so that $R$ is the endomorphism algebra of the tilting bundle $\ca{O} \oplus \ca{O}(-1)$ on $\mb{P}^1$. In particular, we get $D(R)^c \simeq D^b(\mb{P}^1)$, and the entropy of the Serre functor is $t$.
	In this case, \autoref{lem:inequalities} applies for any $d \geq 1$.

	When $w \neq 0$ we will need to appeal to \autoref{prop:bound-entropies} and continuity of entropy, see \cite[Theorem 2.1.6]{Shifting-numbers}, and for this reason we need to restrict to $\mcr{T} = D(T)^c$ for $T$ a compact and smooth dg algebra.

	All the statements follow easily from \autoref{lem:entropy-at-0-Serre}, \autoref{prop:bound-entropies}, and \autoref{lem:asymptotic-behaviour} below.
\end{proof}

\begin{lem}
	\label{lem:asymptotic-behaviour}
	If $\mcr{T} = D(T)^c$ for a compact, smooth dg algebra, $\dim V \geq 2$, and $\mathrm{Hom}_{D^b(k)}(V,V[d]) = 0$, we have
	\[
	\begin{array}{c}
	\displaystyle{\lim_{t \to -\infty} \frac{h_t(T_{2} \circ T_{1})}{t}} = 2 - (d+w) \quad \text{if} \; d+w \geq 2,
	\\
	\displaystyle{\lim_{t \to +\infty} \frac{h_t(T_{2} \circ T_{1})}{t}} = 2 - (d-w) \quad \text{if} \; d-w \leq 2,
	\end{array}
	\]
	and they are $\geq 0$ and $\leq 0$ otherwise.
	Here we set $w = \max V - \min V$.
\end{lem}

\begin{proof}
	The assumptions, together with \autoref{thm:entropy-twist-spherical-functor} and \autoref{lem:inequalities}, imply that\footnote{Here $R$ depends on how much we shift $E_1$, but the limit won't, so we drop the dependence on $n$.}
	\[
		h_t(T_{2} \circ T_{1}) = h_t(\mb{S}_R) + (1-d)t
	\]
	as long as the right hand side is bigger or equal than $0$, and $h_t(T_{2} \circ T_{1}) \leq 0$ otherwise.
	Dividing the right hand side by $t$ and passing to the limit for $t \to -\infty$, by \cite[Proposition 8.4]{Elagin-examples-dimensions} we see that the right hand side tends to $2 - (d+w)$.
	If $d + w \geq 2$ by continuity of the entropy of an endofunctor we get that there exists an $N < 0$ such that the right hand side above is positive in $(-\infty, N)$.
	Hence, the equality holds in this interval, and the statement follows dividing by $t$ and passing to the limits.
	If $d+w < 2$ then in $(-\infty, N)$ the right hand side is negative and we only get the inequality.

	Similarly one proves the statement for $t \to +\infty$.
\end{proof}

\section{Counterexamples to Kikuta--Takahashi}



In this section, using \autoref{thm:entropy-composition-1}, we will produce new counterexamples to Kikuta--Takahashi's conjecture, \cite{Kikuta-Takahashi-On-cat-entropy}.
In particular, we will produce the first counterxamples in odd dimension.

Let $\mathscr{T}$ be a $k$-linear, proper, dg enhanced triangulated category with a Serre functor and a split generator, and let $K(\mathscr{T})$ be its Grothendieck group.
The {\em Euler form} $\chi : K(\mathscr{T}) \times K(\mathscr{T}) \to \mathbb{Z}$ is defined by
\begin{equation*}
	\chi([E],[F]) = \sum_{i \in \mathbb{Z}} (-1)^i \dim\mathrm{Hom}_\mathscr{T}(E,F[i]).
\end{equation*}
We define the {\em numerical Grothendieck group} $K_\mathrm{num}(\mathscr{T})$ as\footnote{Notice that the existence of a Serre functor implies that the right and left radical of $\chi$ agree, so there is no ambiguity in the definition of $K_\mathrm{num}(\mcr{T})$.}
\begin{equation*}
	K_\mathrm{num}(\mathscr{T}) = K(\mathscr{T})/\langle [E] \in K(\mathscr{T}) \,|\, \chi([E],-) = 0 \rangle.
\end{equation*}
Note that the induced Euler form $\chi : K_\mathrm{num}(\mathscr{T}) \times K_\mathrm{num}(\mathscr{T}) \to \mathbb{Z}$ is non-degenerate.
In this section, we only consider triangulated categories whose numerical Grothendieck groups are of finite rank.

\begin{cor}
	\label{cor:examples}
	Let $E_1,E_2 \in \mathscr{T}$ be $d$-spherical objects and $V = \mathrm{Hom}_\mathscr{T}^\bullet(E_2,E_1)$.
	Suppose $[E_1],[E_2]$ are non-zero and linearly independent in $K_\mathrm{num}(\mathscr{T})$, that $\mathrm{Hom}_{D^b(k)}(V,V[d]) = 0$, and that $\chi([E_2],[E_1]) \neq \pm 2$ if $d$ is even, $\chi([E_2], [E_1]) \neq 0$ if $d$ is odd.
	If $\dim V = 0,1,2$, then
	\[
		h_0(T_2 \circ T_1) = \log\rho([T_2 \circ T_1]) = 0,
	\]
	and if $\dim V \geq 3$, then
	\[
		h_0(T_2 \circ T_1) \geq \log\rho([T_2 \circ T_1]),
	\]
	where the equality holds if and only if $\chi([E_2],[E_1]) = \pm \dim V$.
\end{cor}

\begin{proof}
	First of all, notice that as $[E_1],[E_2]$ are assumed to be non-zero and linearly independent the subspace $W:= k \{[E_1],[E_2]\}$ is two dimensional.
	Moreover, the fact that $[E_1],[E_2]$ are $d$-spherical implies $W^{\perp} = {}^{\perp} W$.

	The assumptions on $\chi([E_2], [E_1])$ imply that the restriction of $\chi$ to this subspace is non-degenerate.
	Hence, by what we said above, we get a basis of $K_\mathrm{num}(\mathscr{T})$ by taking $[E_1]$,$[E_2]$ and a basis of $k \{[E_1], [E_2]\}^{\perp}$.

	By definition, we have
	\[
		[T_i](v) = v - \chi([E_i],v)[E_i] \quad v \in K_\mathrm{num}(\mathscr{T}).
	\]
	Denote $\lambda = \chi([E_2],[E_1])$.
	Then, with respect to the previously chosen basis,
	\[
		[T_2 \circ T_1] =
		\left(
		\begin{array}{cc}
			A & 0\\
			0 & \mathrm{Id}
		\end{array}
		\right)
	\]
	where
	\[
		A = (-1)^{1-d}
		\left(
		\begin{array}{cc}
			1 & \lambda\\
			-\lambda & 1 - \lambda^2
		\end{array}
		\right).
	\]

	The eigenvalues are
	\[
		\displaystyle{\frac{\lambda^2 -2 \pm \sqrt{\lambda^4 - 4\lambda^2}}{2}},
	\]
	and then logarithm of the spectral radius is
	\[
		\log\rho([T_2 \circ T_1]) = 
		\log\left| \frac{\lambda^2-2 + \sqrt{\lambda^4-4\lambda^2}}{2} \right|
	\]

	By \autoref{thm:entropy-composition-1}, we have\footnote{Notice that here we are taking the absolute value of (possibly) a complex number.}
	\[
		h_0(T_2 \circ T_1) = \displaystyle{\log \left| \frac{(\dim V)^2 -2 + \sqrt{(\dim V)^4 - 4(\dim V)^2}}{2} \right|}.
	\]

	This shows the statement for $\dim V = 0,1,2$.
	As the function
	\[
		x \mapsto \displaystyle{\log \left( \frac{x -2 + \sqrt{x^2 - 4x}}{2} \right)}
	\]
	is injective on $x \geq 4$, we also get the statement for $\dim V \geq 3$.
\end{proof}

\begin{ex}
	\label{ex:KT-PxP}
	Consider $\mb{P}^n \times \mb{P}^{m}$ with either
	\begin{enumerate}
		\item $n \geq 3$, $n$ odd, $m \geq 2$, $m$ even;
		\item $n,m \geq 2$, $n,m$ even.
	\end{enumerate}
	Take $X$ to be the zero locus of a section of $\ca{O}_{\mb{P}^n \times \mb{P}^m}(n+1,m+1)$.
	Then, from the exact sequence
	\[
		\begin{tikzcd}
			\ca{O}_{\mb{P}^n \times \mb{P}^m}(-n-1,-m-1) \ar[r] & \ca{O}_{\mb{P}^n \times \mb{P}^m} \ar[r] & \ca{O}_X
		\end{tikzcd}
	\]
	we see that $X$ is a true Calabi--Yau manifold of dimension $m+n-1$.
	In particular, line bundles on $X$ are $d := m+n-1$ spherical objects.

	Consider $\ca{L} = \ca{O}_X(n+1,0)$.
	Then, from the above exact sequence we see that
	\[
		R\Gamma(\ca{L}) \simeq k^{N} \oplus k[-m+1] \quad N = \left( \begin{array}{c} 2n+1 \\ n+1\end{array}\right)
	\]
	In particular, if we set $V = \text{RHom}_{X}(\ca{O}_X, \ca{L})$, we have: $\dim V = N +1$, $\lambda = N - 1 > 2$ in case (1) ($m$ is even and $n \geq 3$), and $\lambda = N - 1 > 0$ in case (2) ($m$ is even and $n \geq 2$).
	Moreover, we have $\max V - \min V = m-1 < d$, and therefore $V$ doesn't have a gap of $d$.

	As the line bundles $\ca{O}_X$ and $\ca{L}$ have linearly independent Mukai vectors,\footnote{They are line bundles with different first Chern class.} \autoref{cor:examples} applies.
	In particular, we get
	\[
		\begin{aligned}
			\log \left( \rho ( T_{\ca{O}_X}^H \circ T_{\ca{L}}^H )\right) = &\log \left( \frac{(N-1)^2 - 2 + \sqrt{(N-1)^4 - 4(N-1)^2}}{2} \right)\\
			< &\log \left( \frac{(N+1)^2 - 2 + \sqrt{(N+1)^4 - 4(N+1)^2}}{2} \right) = h_{0} \left( T_{\ca{O}_X} \circ T_{\ca{L}} \right),
		\end{aligned}
	\]
	thus contradicting Kikuta--Takahashi's conjecture.
\end{ex}

\subsection{\texorpdfstring{$A_2$ Ginzburg dg algebra}{A2 Ginzburg}}
\label{subs:A_2}

The {\em $d$-Calabi--Yau Ginzburg dg algebra} $\Gamma_2^d$ associated to the $A_2$ quiver is defined as follows.
First, as a graded algebra, it is the path algebra of the graded quiver with two vertices $\{1,2\}$ and four arrows: $a : 1 \to 2$ in degree $0$, $a^* : 2 \to 1$ in degree $2-d$ and $t_i : i \to i$ ($i = 1,2$) in degree $1-d$.
The differential is given by $da = da^* = 0$, $dt_1 = aa^*$ and $dt_2 = -a^*a$.

Let $\mathscr{D}_2^d$ be the derived category of dg $\Gamma_2^d$-modules with finite dimensional cohomology.
It is known that $\mathscr{D}_2^d$ is $d$-Calabi--Yau category and the simple modules $S_1,S_2$ are spherical objects such that $V = \mathrm{Hom}^\bullet(S_2,S_1) = \mathbb{C}[1-d]$.
Denote by $T_1,T_2$ the spherical twists around them; we obtain a braid group action via $\mathrm{Br}_3 = \langle \sigma_1,\sigma_2 \rangle \ni \sigma_i \mapsto T_i$.
We call an object a {\em reachable spherical object} if it is isomorphic to an object $\sigma S_i$ for some $\sigma \in \mathrm{Br}_3$ and $i=1,2$.
For two reachable spherical objects $E_1,E_2$, the Poincar\'e polynomial of $\mathrm{Hom}_{\mathscr{D}_2^d}^\bullet(E_2,E_1)$, {\it i.e.}
\[
p(E_2,E_1) = \sum_{n \in \mathbb{Z}} \dim \mathrm{Hom}_{\mathscr{D}_2^d}(E_2,E_1[n]) q^n
\]
coincides with a weighted intersection number of some arcs on the disk with 3 marked points, \cite{Khovanov-Seidel01}.
Let us recall the precise statement.

Let $(D,\Delta)$ be the unit disk $D$ with 3 marked points $\Delta = \{ p_1,p_2,p_3 \} \subset D$.
A {\em closed arc} in $(D,\Delta)$ is an embedding $c : [0,1] \to D$ such that $c^{-1}(\Delta) = \{0,1\}$.
Define $P = \mathbb{P}(T(D \setminus \Delta))$ to be the real projectivization of the tangent bundle of $D \setminus \Delta$.
By considering an oriented trivialization of $D$, we can identify $P$ with $\mathbb{RP}^1 \times (D \setminus \Delta)$.
For each $p_i$, take a small loop $\lambda_i$ winding $p_i$ positively once.
Then $[\mathrm{pt} \times \lambda_i]$ and $[\mathbb{RP}^1 \times \mathrm{pt}]$ form a basis of $H_1(P,\mathbb{Z})$.
Define $\alpha \in H^1(P,\mathbb{Z}^2)$ by $\alpha([\mathrm{pt} \times \lambda_i]) = (-2,1)$ and $\alpha([\mathbb{RP}^1 \times \mathrm{pt}]) = (1,0)$.
Let $\tilde{P}$ be the covering space with covering group $\mathbb{Z}$ determined by $\alpha$.
A {\em bigraded closed arc} $(c,\tilde{c})$ (or $\tilde{c}$ for short) in $(D,\Delta)$ is a closed arc $c$ in $(D,\Delta)$ together with a lift $\tilde{c} : (0,1) \to \tilde{P}$ of the section $s_c : (0,1) \to P$ given by $s_c(t) = T_{c(t)}c$.

Let $\tilde{c}_0,\tilde{c}_1$ be bigraded closed arcs having minimal intersection in the sense that they intersect transversely and do not bound a disk.
We shall define a bigrading of an intersection point $z \in c_0 \cap c_1$.
Take a small loop $l$ around $z$ and an arc $a : [0,1] \to l \subset D$ which moves clockwise along $l$ and $a^{-1}(c_i) =\{i\}$ for $i=0,1$.
Let us also take a path $\pi : [0,1] \to P$ such that $\pi(t) \in P_{a(t)}$ for all $t$, $\pi(i) = T_{a(i)}c_i$ for $i=0,1$ and $\pi(t) \neq T_{a(t)}l$ for all $t$.
Let $\tilde{\pi} : [0,1] \to \tilde{P}$ be the lift of $\pi$ with $\tilde{\pi}(0) = \tilde{c}_0(a(0))$.
Then we have $\tilde{c}_1(a(1)) = (\mu_1,\mu_2) \cdot \tilde{\pi}(1)$ for a unique $(\mu_1,\mu_2) \in \mathbb{Z}^2$ which acts as a covering transformation.
In this case, we denote $(\mu_1(z),\mu_2(z)) = (\mu_1,\mu_2)$ and define the {\em bigraded intersection number} of $\tilde{c}_0$ and $\tilde{c}_1$ to be
\[
I(\tilde{c}_0,\tilde{c}_1) = (1+q_1^{-1}q_2) \sum_{z \in (c_0 \cap c_1) \setminus \Delta} q_1^{\mu_1(z)}q_2^{\mu_2(z)} + \sum_{z \in c_0 \cap c_1 \cap \Delta} q_1^{\mu_1(z)}q_2^{\mu_2(z)} \in \mathbb{Z}[q_1^{\pm 1},q_2^{\pm 1}].
\]

It was shown in \cite{Khovanov-Seidel01} that the behavior of reachable spherical objects of $\mathscr{D}_2^d$ can be read off from the topology of bigraded closed arcs in $(D,\Delta)$.
More precisely, there are some bigraded closed arcs $\tilde{b}_1,\tilde{b}_2$ and a braid group action $\mathrm{Br}_3 = \langle \sigma_1,\sigma_2 \rangle \ni \sigma_i \mapsto t_i$, where $t_i$ is the half twist around $b_i$, satisfying
\[
p(\sigma S_i,\tau S_j) = I(\sigma \tilde{b}_i,\tau \tilde{b}_j)|_{q_1=q,q_2=q^d}\footnote{We only need the existence of such bigraded closed arcs, for their explicit description see \cite{Khovanov-Seidel01}.}
\]
for any $\sigma,\tau \in \mathrm{Br}_3$ and $i,j=1,2$.

\begin{cor}\label{cor:ginzburg}
	Let $E_1,E_2$ be reachable spherical objects  which are non-isomorphic to each other.
	Then self extensions of $V = \mathrm{Hom}_{\mathscr{D}_2^d}^\bullet(E_2,E_1)$ have degree of the form $k(d-1)$ for some $k \geq 0$.
	In particular, we have
	\[
		h_0(T_2 \circ T_1) = \displaystyle{\log \left| \frac{(\dim V)^2 -2 + \sqrt{(\dim V)^4 - 4(\dim V)^2}}{2} \right|}.
	\]
	Moreover,
	\[
		h_0(T_2 \circ T_1) = \log\rho([T_2 \circ T_1])
	\]
	holds if and only if $\dim V = 1,2$ or $\dim V \geq 3$ and $d$ is odd.
\end{cor}

\begin{proof}
For simplicity, we assume that the marked points on the disk are $p_1=(-\frac{1}{2},0),p_2=(0,0),p_3=(\frac{1}{2},0)$.
Without loss of generality, we can assume that $E_2 = S_2$ and a bigraded closed arc $\tilde{b}_2$ corresponding to it is the straight arc connecting $p_2$ and $p_3$.
Let $\tilde{c}$ be a bigraded closed arc correponding to $E_1$.
By twisting around $b_2$, we can assume $c$ has $p_2$ as one of its end points.
Denote the intersection points of $b_2$ and $c$ by $z_i = (a_i,0)$ where $0=a_1<a_2<\cdots<a_n\leq\frac{1}{2}$.

To prove the first claim, it is enough to prove that
\[
\delta_i = \mu_1(z_{i+1})+d\mu_2(z_{i+1}) - (\mu_1(z_i)+d\mu_2(z_i)) = k(d-1)
\]
for some $k \in \mathbb{Z}$ and for all $i$.
Here, we shall only show it for the case $i=1$ as the other cases can be shown similarly.
Up to twisting around $b_2$,\footnote{Note that twisting around $b_2$ doesn't change $\delta_i$.} we have four possibilities in that case which are depicted in Figure 1.
For each of four cases, $\delta_1$ is $d-1,2(d-1),-2(d-1)$ and $-3(d-1)$ respectively.

\begin{figure}[t]
\begin{tikzpicture}
\filldraw (-5.5,0) circle (2pt);
\filldraw (-3.5,0) circle (2pt);
\filldraw (-1.5,0) circle (2pt);
\draw (-3.5,0) -- (-1.5,0);
\draw (-3.5,0) .. controls (-4,1) and (-6.5,1) .. (-6.5,0) .. controls (-6.5,-1.5) and (-2.5,-1) .. (-2.5,0.5);

\filldraw (1.5,0) circle (2pt);
\filldraw (3.5,0) circle (2pt);
\filldraw (5.5,0) circle (2pt);
\draw (3.5,0) -- (5.5,0);
\draw (3.5,0) .. controls (3,1) and (0.5,1) .. (0.5,0) .. controls (0.5,-1.5) and (6.5,-1.5) .. (6.5,0) .. controls (6.5,1) and (4.5,1) .. (4.5,-0.5);

\filldraw (-5.5,-3) circle (2pt);
\filldraw (-3.5,-3) circle (2pt);
\filldraw (-1.5,-3) circle (2pt);
\draw (-3.5,-3) -- (-1.5,-3);
\draw (-3.5,-3) .. controls (-4,-4) and (-6.5,-4) .. (-6.5,-3) .. controls (-6.5,-1.5) and (-2.5,-2) .. (-2.5,-3.5);

\filldraw (1.5,-3) circle (2pt);
\filldraw (3.5,-3) circle (2pt);
\filldraw (5.5,-3) circle (2pt);
\draw (3.5,-3) -- (5.5,-3);
\draw (3.5,-3) .. controls (3,-4) and (0.5,-4) .. (0.5,-3) .. controls (0.5,-1.5) and (6.5,-1.5) .. (6.5,-3) .. controls (6.5,-4) and (4.5,-4) .. (4.5,-2.5);
\end{tikzpicture}
\caption{$\delta_1$ is $d-1$ (top left), $2(d-1)$ (top right), $-2(d-1)$ (bottom left) and $-3(d-1)$ (bottom right) respectively.}
\end{figure}
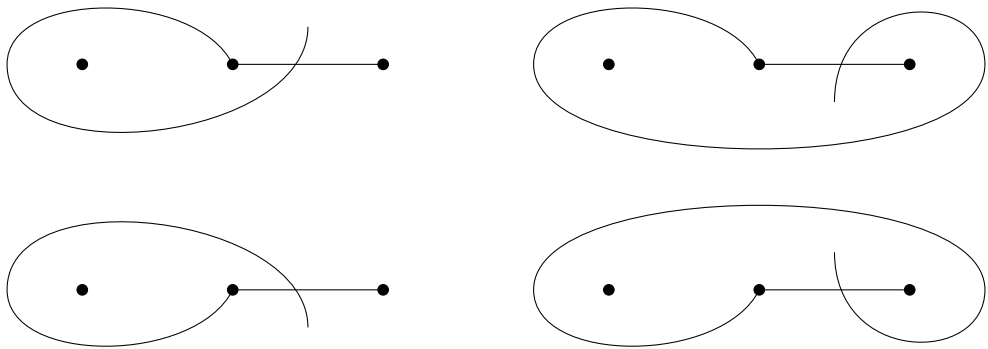

The second claim can be seen by noticing that $\chi([E_2],[E_1]) = \dim V$ when $d$ is odd while $\chi([E_2],[E_1]) = 1,2$ when $d$ is even.\footnote{When $d$ is even, the absolute value of $\chi([E_2],[E_1])$ is exactly the number of common end points of the closed arcs corresponding to $E_1,E_2$.}
\end{proof}

This example has the following symplecto-geometric interpretation.
Let $X_2^d$ be the Milnor fiber of $A_2$-singularity of dimension $2d>2$ and $L_1,L_2$ be the vanishing cycles (equipped with suitable grading structures).
It is known that $L_1,L_2$ split-generate the (split-closed) derived Fukaya category $D^\pi\mathcal{F}(X_2^d)$.
Since $S_1,S_2$ also split-generate the finite-dimensional derived category $\mathscr{D}_2^d$ and the graded algebra
\begin{equation*}
\bigoplus_{i,j=1}^2 \mathrm{Hom}^\bullet_{D^\pi\mathcal{F}(X_2^d)}(L_i,L_j) \cong \bigoplus_{i,j=1}^2 \mathrm{Hom}^\bullet_{\mathscr{D}_2^d}(S_i,S_j)
\end{equation*}
is intrinsically formal (see \cite[Lemma 4.21]{Seidel-Thomas01}), we have an exact equivalence
\begin{equation*}
D^\pi\mathcal{F}(X_2^d) \simeq \mathscr{D}_2^d.
\end{equation*}
In particular, under this equivalence, $L_i$ corresponds to $S_i$ and the Dehn twist $\tau_i$ around $L_i$ corresponds to the spherical twist $T_i$ around $S_i$.
Therefore, \autoref{cor:ginzburg} can be stated in terms of symplectic geometry.

\begin{cor}\label{cor:milnor}
Let $L_1,L_2$ be reachable Lagrangian spheres in $X_2^d$.
Then we have
\[
h_0(\tau_2 \circ \tau_1) = \displaystyle{\log \left| \frac{m^2 -2 + \sqrt{m^4 - 4m^2}}{2} \right|}
\]
where $m = \dim HF^\bullet(L_2,L_1)$.
\end{cor}

\begin{rmk}\label{rmk:torelli}
Let $d$ be even and $\tilde{b}_1,\tilde{b}_2$ be the bigraded closed arcs corresponding to $L_1,L_2$ respectively.
Suppose $\tilde{b}_1$ and $\tilde{b}_2$ share only one end point.
Then, $p(L_1,L_2) = I(\tilde{b}_1,\tilde{b}_2)|_{q_1=1,q_2=q^d}$ implies that $\lambda = \chi(L_1,L_2) = \pm 1$.
Thus, by the Picard--Lefschetz formula, $(\tau_2 \circ \tau_1)^3$ acts on $H_d(X_2^d,\mathbb{Z}) = \langle [L_1],[L_2] \rangle$ as
	\[
		\left(
		\begin{array}{cc}
			-1 & \mp 1\\
			\pm 1 & 0
		\end{array}
		\right)^3
		=
		\left(
		\begin{array}{cc}
			1 & 0\\
			0 & 1
		\end{array}
		\right),
	\]
{\it i.e.} it is in the symplectic Torelli group of $X_2^d$.
As we have seen, the categorical entropy of $\tau_2 \circ \tau_1$ (and also $(\tau_2 \circ \tau_1)^3$) is positive whenever $\dim HF^\bullet(L_2,L_1) \geq 3$.
Therefore, in such a case, $(\tau_2 \circ \tau_1)^3$ gives a higher-dimensional counterexample to Kikuta--Takahashi's conjecture coming from an element in the symplectic Torelli group having positive categorical entropy.
This answers a question in \cite[Problem 1.2]{Kikuta-Ouchi} about the existence of such an autoequivalence for higher-dimensions.
\end{rmk}


\bibliography{Bibliography-entropy}
\bibliographystyle{alphaurl}

\end{document}